\documentclass{amsart}
\author{Dilip Raghavan}
\thanks{First author partially supported by National University of Singapore 
research grant numbers R-146-000-161-133 and R-146-000-211-112.}
\address{Department of Mathematics \\
National University of Singapore\\
Singapore 119076}
\email{dilip.raghavan@protonmail.com}
\urladdr{http://www.math.toronto.edu/raghavan}
\author{Saharon Shelah}
\thanks{Both authors were partially supported by European Research Council grant 338821. Paper 1060 on Shelah's list.}
\date{\today}
\subjclass[2010]{03E17, 03E55, 03E05, 03E20}
\keywords{asymptotic density, cardinal invariants, dominating number, weakly compact cardinal}
\title[Two inequalities]{Two inequalities between cardinal invariants}
\usepackage{amssymb, amsmath, amsthm, mathrsfs, enumerate, amsfonts, latexsym, bbm}
\def\polhk#1{\setbox0=\hbox{#1}{\ooalign{\hidewidth
    \lower1.5ex\hbox{`}\hidewidth\crcr\unhbox0}}}
\newtheorem{Theorem}{Theorem}

\newtheorem{Lemma}[Theorem]{Lemma}
\newtheorem{Cor}[Theorem]{Corollary}

\newtheorem{Question}[Theorem]{Question}

\theoremstyle{definition}
\newtheorem{Def}[Theorem]{Definition}

\theoremstyle{remark}

\newcommand{\forces}{\Vdash}
\newcommand{\restrict}{\upharpoonright}
\newcommand{\forallbutfin}{{\forall}^{\infty}}

\renewcommand{\b}{\mathfrak{b}}
\renewcommand{\d}{{\mathfrak{d}}}
\newcommand{\s}{\mathfrak{s}}

\newcommand{\p}{{\mathfrak{p}}}
\renewcommand{\a}{{\mathfrak{a}}}

\renewcommand{\[}{\left[}
\renewcommand{\]}{\right]}
\renewcommand{\P}{\mathbb{P}}

\newcommand{\Q}{\mathbb{Q}}
\newcommand{\lc}{\left|}
\newcommand{\rc}{\right|}

\newcommand\ZFC{\mathrm{ZFC}}
\newcommand\FIN{\mathrm{FIN}}

\newcommand{\BS}{{\omega}^{\omega}}
\DeclareMathOperator{\non}{non}

\DeclareMathOperator{\tr}{tr}

\DeclareMathOperator{\otp}{otp}
\DeclareMathOperator{\cov}{cov}
\DeclareMathOperator{\add}{add}
\DeclareMathOperator{\cof}{cof}

\DeclareMathOperator{\cf}{cf}

\newcommand{\Pset}{\mathcal{P}}
\newcommand{\MMM}{\mathcal{M}}
\newcommand{\NNN}{\mathcal{N}}

\newcommand{\ZZZ}{{\mathcal{Z}}}

\newcommand{\cube}{{\[\omega\]}^{\omega}}

\newcommand{\I}{{\mathcal{I}}}
\newcommand{\JJJ}{{\mathcal{J}}}

\newcommand{\F}{{\mathcal{F}}}

\newcommand{\V}{{\mathbf{V}}}

\newcommand{\VP}{{\mathbf{V}}^{\P}}

\newcommand{\pr}[2]{\langle #1, #2 \rangle}
\newcommand{\seq}[4]{\langle {#1}_{#2}: #2 #3 #4 \rangle}
\begin{document}
\begin{abstract}
We prove two $\ZFC$ inequalities between cardinal invariants.
The first inequality involves cardinal invariants associated with an analytic P-ideal, in particular the ideal of subsets of $\omega$ of asymptotic density $0$.
We obtain an upper bound on the $\ast$-covering number, sometimes also called the weak covering number, of this ideal by proving in Section \ref{sec:covz0} that ${\cov}^{\ast}({\ZZZ}_{0}) \leq \d$.
In Section \ref{sec:skbk} we investigate the relationship between the bounding and splitting numbers at regular uncountable cardinals.
We prove in sharp contrast to the case when $\kappa = \omega$, that if $\kappa$ is any regular uncountable cardinal, then ${\s}_{\kappa} \leq {\b}_{\kappa}$. 
\end{abstract}
\maketitle
\section{Introduction} \label{sec:intro}
Cardinal invariants associated with analytic P-ideals and their quotients are becoming increasingly well studied.
Several cardinal invariants have been defined and investigated for quotients of the form $\Pset(\omega) \slash \I$, where $\I$ is some definable ideal, guided by analogy with the familiar case of the quotient $\Pset(\omega) \slash \FIN$.
The most interesting among these have been the cases where $\I$ is either ${F}_{\sigma}$ or an analytic P-ideal.
Recall that an ideal $\I$ on $\omega$ is called a \emph{P-ideal} if for every collection $\{{a}_{n}: n \in \omega\} \subset \I$, there exists $a \in \I$ such that $\forall n \in \omega \[{a}_{n} \; {\subset}^{\ast} \; a\]$.
Here $a \; {\subset}^{\ast} \; b$ means that $a \setminus b$ is finite.
When $\I$ is a tall ideal, it is possible to associate some cardinals with $\I$ that it wouldn't necessarily make sense to do with $\FIN$.
Recall that an ideal $\I$ on $\omega$ is \emph{tall} if it is \emph{proper}, meaning $\omega \notin \I$, it is \emph{non-principal}, meaning that ${\[\omega\]}^{< \omega} \subset \I$, and it has the property that $\forall a \in \cube \exists b \in {\[a\]}^{\omega}\[b \in \I\]$.
\begin{Def} \label{def:starinv}
 Let $\I$ be a tall P-ideal on $\omega$.
 Define
 \begin{align*}
  &{\add}^{\ast}(\I) = \min\{\lc \F \rc: \F \subset \I \wedge \forall b \in \I \exists a \in \F \[a \; {\not\subset}^{\ast} \; b\]\},\\
  &{\cov}^{\ast}(\I) = \min\{\lc \F \rc: \F \subset \I \wedge \forall a \in \cube \exists b \in \F\[\lc a \cap b \rc = \omega\]\},\\
  &{\cof}^{\ast}(\I) = \min\{\lc \F \rc: \F \subset \I \wedge \forall b \in \I \exists a \in \F \[b \; {\subset}^{\ast} \; a\]\},\\
  &{\non}^{\ast}(\I) = \min\{\lc \F \rc: \F \subset \cube \wedge \forall b \in \I\exists a \in \F\[\lc a \cap b \rc < \omega\]\}.
 \end{align*}
\end{Def}
Cardinals of this kind were first considered by Brendle and Shelah~\cite{bs1} and by Bartoszy{\'n}ski~\cite{barsmall}.
Brendle and Shelah~\cite{bs1} referred to ${\add}^{\ast}(\I)$ as $\p({\I}^{\ast})$ and ${\cov}^{\ast}(\I)$ as $\pi\p({\I}^{\ast})$, where ${\I}^{\ast} = \{\omega \setminus a: a \in \I\}$ is the \emph{dual filter} to $\I$.
The present terminology was popularized by Hern{\'a}ndez-Hern{\'a}ndez and Hru{\v{s}}{\'a}k~\cite{michaelstarinv} who carried out a detailed investigation of these invariants for tall analytic P-ideals.
Their choice of terminology was motivated by analogy with the following definitions which make sense for any ideal whatsoever.
\begin{Def} \label{def:Ichichon}
 Let $\I$ be any ideal on a set $X$. Define
 \begin{align*}
  &\add(\I) = \min\left\{\lc \F \rc: \F \subset \I \wedge \bigcup\F \notin \I\right\},\\
  &\cov(\I) = \min\left\{\lc \F \rc: \F \subset \I \wedge \bigcup\F = X\right\},\\
  &\cof(\I) = \min\left\{\lc \F \rc: \F \subset \I \wedge \forall B \in \I \exists A \in \F\[B \subset A\]\right\},\\
  &\non(\I) = \left\{\lc Y \rc: Y \subset X \wedge Y \notin \I\right\}.
 \end{align*}
\end{Def}
It is possible to associate with each tall ideal $\I$ on $\omega$ an ideal $\hat{\I}$ on $\Pset(\omega)$ which is generated by Borel subsets of $\Pset(\omega)$ in a natural way.
For each $a \in \Pset(\omega)$, let $\hat{a} = \{b \subset \omega: \lc a \cap b \rc = \omega\}$.
This is a ${G}_{\delta}$ subset of $\Pset(\omega)$.
If $\I$ is a tall ideal on $\omega$, then $\hat{\I} = \left\{X \subset \Pset(\omega): \exists a \in \I\[X \subset \hat{a}\]\right\}$ is an ideal on $\Pset(\omega)$ generated by Borel sets.
Moreover $\I$ is a P-ideal iff $\hat{\I}$ is a $\sigma$-ideal.
Now the invariants from Definition \ref{def:Ichichon} associated with $\hat{\I}$ correspond exactly with the $\ast$-invariants from Definition \ref{def:starinv} associated with $\I$.
It can be shown (see Proposition 1.2 of \cite{michaelstarinv}) that $\add(\hat{\I}) = {\add}^{\ast}(\I)$, $\cov(\hat{\I}) = {\cov}^{\ast}(\I)$, $\cof(\hat{\I}) = {\cof}^{\ast}(\I)$, $\non(\hat{\I}) = {\non}^{\ast}(\I)$.

One of the main tools used in \cite{michaelstarinv} for analyzing the $\ast$-invariants of tall analytic P-ideals is the Kat{\v e}tov order.
Let $\I$ and $\JJJ$ be ideals on $\omega$.
Recall that $\I$ is \emph{Kat{\v e}tov below} $\JJJ$ or $\I \; {\leq}_{K} \; \JJJ$ if there is a function $f: \omega \rightarrow \omega$ such that $\forall a \in \I\[{f}^{-1}(a) \in \JJJ\]$.
The significance of this ordering lies in the fact that $\I \; {\leq}_{K} \; \JJJ$ implies both that ${\cov}^{\ast}(\I) \geq {\cov}^{\ast}(\JJJ)$ and that ${\non}^{\ast}(\I) \leq {\non}^{\ast}(\JJJ)$ (see Proposition 3.1 of \cite{michaelstarinv}).
The Tukey ordering is also relevant here.
We say that $\pr{\I}{{\subset}^{\ast}}$ is \emph{Tukey below} $\pr{\JJJ}{{\subset}^{\ast}}$ and we write $\I \; {\leq}^{\ast}_{T} \; \JJJ$ if there is a map $\varphi: \I \rightarrow \JJJ$ such that if $X \subset \I$ any set that does not have an upper bound in the partial order $\pr{\I}{{\subset}^{\ast}}$, then $\varphi''X$ does not have an upper bound in the partial order $\pr{\JJJ}{{\subset}^{\ast}}$.
If $\I \; {\leq}^{\ast}_{T} \; \JJJ$, then ${\add}^{\ast}(\I) \geq {\add}^{\ast}(\JJJ)$ and ${\cof}^{\ast}(\I) \leq {\cof}^{\ast}(\JJJ)$ (see Proposition 2.1 of \cite{michaelstarinv}).
The ideal ${\ZZZ}_{0}$ of sets of asymptotic density 0 is a critical object of study in \cite{michaelstarinv}.
\begin{Def} \label{def:z0}
 A set $A \subset \omega$ is said to have \emph{asymptotic density 0} if $\displaystyle\lim_{n \rightarrow \infty}{\displaystyle\frac{\lc A \cap n \rc}{n}} = 0$.
 ${\ZZZ}_{0} = \{A \subset \omega: {\lim}_{n \rightarrow \infty}{\frac{\lc A \cap n \rc}{n}} = 0\}$.
\end{Def}
${\ZZZ}_{0}$ is easily seen to be a tall ${F}_{\sigma\delta}$ P-ideal.
It is pointed out in \cite{michaelstarinv} that ${\add}^{\ast}({\ZZZ}_{0}) = \add(\NNN)$ and ${\cof}^{\ast}({\ZZZ}_{0}) = \cof(\NNN)$, where $\NNN$ is the ideal of subsets of $\mathbb{R}$ that have Lebesgue measure 0.
This follows from earlier work of Todorcevic~\cite{analyticgaps} and Fremlin~\cite{frem3} on the Tukey order ${\leq}^{\ast}_{T}$.
Hern{\'a}ndez-Hern{\'a}ndez and Hru{\v{s}}{\'a}k~\cite{michaelstarinv} prove that ${\mathcal{E}\mathcal{D}}_{\mathrm{fin}} \;{\leq}_{K} \; {\I}_{\frac{1}{n}} \; {\leq}_{K} \; \tr(\NNN)\; {\leq}_{K} \; {\ZZZ}_{0}$ (see \cite{michaelstarinv} for the definitions of the ideals ${\mathcal{E}\mathcal{D}}_{\mathrm{fin}}$, ${\I}_{\frac{1}{n}}$, and $\tr(\NNN)$).
The upshot of this is that $\add(\NNN) \leq {\cov}^{\ast}({\ZZZ}_{0}) \leq \non(\MMM)$ and $\cov(\MMM) \leq {\non}^{\ast}({\ZZZ}_{0}) \leq \cof(\NNN)$, where $\MMM$ is the ideal of meager subsets of $\mathbb{R}$.
They then improve these bounds by proving that
$\min\{\cov(\NNN), \b\} \leq {\cov}^{\ast}({\ZZZ}_{0}) \leq \max\{\b, \non(\NNN)\}$ and that $ \min\{\d, \cov(\NNN)\} \leq {\non}^{\ast}({\ZZZ}_{0}) \leq \max\{\d, \non(\NNN)\}$.
Here $\b$ is the least size of an unbounded family in $\BS$ with respect to the ordering of eventual domination and $\d$ is the least size of a cofinal family.
It is also proved in \cite{michaelstarinv} that ${\ZZZ}_{0}$ is Kat{\v e}tov minimal among all density ideals.
Given these results the following question naturally suggests itself.
\begin{Question}[Question 3.23(a) of \cite{michaelstarinv}]\label{q:main1}
 Is ${\cov}^{\ast}({\ZZZ}_{0}) \leq \d$?
\end{Question}
Apart from the intrinsic interest in locating ${\cov}^{\ast}({\ZZZ}_{0})$ in relation to the cardinals in Cicho{\'n}'s diagram, Question \ref{q:main1} also has a motivation coming from forcing theory.
Recall the following definition.
\begin{Def} \label{def:diagonalize}
 Let $\V$ be any ground model and $\P \in \V$ be a notion of forcing.
 Let $\I \in \V$ be an ideal on $\omega$.
 We say that $\P$ \emph{diagonalizes} $\V \cap \I$ if there exists $\mathring{A} \in \VP$ such that ${\forces}_{\P} \mathring{A} \in \cube$ and for each $X \in \V \cap \I$, ${\forces}_{\P} \lc X \cap \mathring{A} \rc < \omega$.
\end{Def}
If $\I$ is a definable tall ideal and if $\P \in \V$ diagonalizes $\V \cap \I$, then $\P$ tends to push ${\cov}^{\ast}(\I)$ up. 
A classical theorem of Laflamme~\cite{zapping} says that any ${F}_{\sigma}$ ideal can be diagonalized by a proper $\BS$-bounding forcing.
When combined with standard preservation theorems and bookkeeping arguments, Laflamme's result enables the construction of a model where ${\cov}^{\ast}(\I) > \d$ for every tall ${F}_{\sigma}$ ideal $\I$.
Laflamme's result has led to speculation about whether all tall ${F}_{\sigma\delta}$ P-ideals, which arguably constitute the second nicest class of definable ideals after the ${F}_{\sigma}$ ideals, could also be diagonalized by a proper $\BS$-bounding forcing. 
\begin{Question} \label{q:submain1}
Suppose $\I \in \V$ is an ${F}_{\sigma\delta}$ P-ideal.
Does there exist a proper $\BS$-bounding $\P \in \V$ which diagonalizes $\V \cap \I$?
Is it consistent that ${\cov}^{\ast}(\I) > \d$ holds for all tall ${F}_{\sigma\delta}$ P-ideals $\I$?
\end{Question}
It has long been known that one cannot hope for anything like this if one moves up one level to consider ${F}_{\sigma\delta\sigma}$ ideals.
The ideal $\FIN \times \FIN$, which is defined to be $\{a \subset \omega \times \omega: \{n \in \omega: \{m \in \omega: \pr{n}{m} \in a\} \ \text{is infinite}\} \ \text{is finite}\}$, is an ${F}_{\sigma\delta\sigma}$ ideal and any $\P$ that diagonalizes it must add a dominating real.
In Section \ref{sec:covz0} we give a negative answer to both Questions \ref{q:main1} and \ref{q:submain1} by proving the following theorem.
\begin{Theorem} \label{thm:main1}
 ${\cov}^{\ast}({\ZZZ}_{0}) \leq \d$.
\end{Theorem}

Section \ref{sec:skbk} deals with cardinal invariants on uncountable cardinals. 
\begin{Def} \label{def:skbk}
 Let $\kappa > \omega$ be a regular cardinal.
 For $f, g \in {\kappa}^{\kappa}$,  $f \; {<}^{\ast}\; g$ means that $\lc \{\alpha < \kappa: g(\alpha) \leq f(\alpha)\} \rc < \kappa$.
 A set $F \subset {\kappa}^{\kappa}$ is said to be \emph{unbounded} if there does not exist $g \in {\kappa}^{\kappa}$ such that $\forall f \in F\[f \; {<}^{\ast} \; g\]$.
 A set $F \subset {\kappa}^{\kappa}$ is said to be \emph{dominating} if $\forall f \in {\kappa}^{\kappa} \exists g \in F \[f \; {<}^{\ast} \; g\]$.
 
 For $a, b \in \Pset(\kappa)$, we write $a \; {\subset}^{\ast} \; b$ to mean that $\lc a \setminus b \rc < \kappa$.
 Since $\kappa$ is regular, this is equivalent to saying that $\exists \delta < \kappa\[\left( a \setminus \delta \right) \subset b\]$.
 For $a, b \in \Pset(\kappa)$ we say that $a$ \emph{splits} $b$ if both $b \cap a$ and $b \cap \left( \kappa \setminus a \right)$ have cardinality $\kappa$.
 A family $F \subset \Pset(\kappa)$ is called a \emph{splitting family} if $\forall b \in {\[\kappa\]}^{\kappa} \exists a \in F \[a \ \text{splits} \ b\]$.
 
 We define the cardinal invariants ${\b}_{\kappa}$, ${\d}_{\kappa}$, and ${\s}_{\kappa}$ as follows:
 \begin{align*}
  &{\b}_{\kappa} = \min\{\lc F \rc: F \subset {\kappa}^{\kappa} \wedge F \ \text{is unbounded}\};\\
  &{\d}_{\kappa} = \min\{\lc F \rc: F \subset {\kappa}^{\kappa} \wedge F \ \text{is dominating}\};\\
  &{\s}_{\kappa} = \min\{\lc F \rc: F \subset \Pset(\kappa) \wedge F \ \text{is a splitting family}\}.
 \end{align*}
 \end{Def}
 These are of course direct analogues of the cardinals $\b, \d$, and $\s$ that play an important role in the theory of cardinal characteristics on $\omega$. 
 Historically, one of the first works to investigate some these higher analogues in depth was the paper \cite{sh541} by Cummings and Shelah.
 They show in that paper that for a regular $\kappa > \omega$, ${\kappa}^{+} \leq \cf({\b}_{\kappa}) = {\b}_{\kappa} \leq \cf({\d}_{\kappa}) \leq {2}^{\kappa}$.
 They also proved in \cite{sh541} that these are essentially the only restrictions on ${\b}_{\kappa}$ and ${\d}_{\kappa}$ that are provable in $\ZFC$.
 In this sense, ${\b}_{\kappa}$ and ${\d}_{\kappa}$ behave in exactly the same way as $\b$ and $\d$.
 
 While the results of Cummings and Shelah~\cite{sh541} do not involve any large cardinals, it has slowly become clear that obtaining consistency results on cardinal invariants at large cardinals is easier than obtaining the same consistency results at accessible cardinals.
 For example a recent work of Garti and Shelah~\cite{uk2k2} proves the consistency of ${\mathfrak{u}}_{\kappa} < {2}^{\kappa}$ at a supercompact cardinal $\kappa$, based on earlier methods introduced by D{\v{z}}amonja and Shelah~\cite{uk2k1}.
 Here ${\mathfrak{u}}_{\kappa}$ is the smallest number of sets needed to generate a uniform ultrafilter on $\kappa$.
 On the other hand, it is completely open whether this situation is consistent at $\kappa = {\omega}_{1}$. 
 
 It is a classical result of Shelah~\cite{b<s} that $\b < \s$ is consistent.
 This result had a lot of impact on the study of cardinal invariants on $\omega$.
 It was the first published application of creature forcing, a method that has subsequently become indispensable to many consistency results on cardinal invariants on $\omega$.
 The importance of this result is one of the motivations for posing the following question.
 \begin{Question} \label{q:main2}
  Is it consistent to have a regular uncountable cardinal $\kappa$ such that ${\b}_{\kappa} < {\s}_{\kappa}$?
 \end{Question}
 The cardinal ${\s}_{\kappa}$ has been investigated by Kamo~\cite{kamo}, Suzuki~\cite{suzukisplitting}, and Zapletal~\cite{jindrasplitting}.
 Suzuki~\cite{suzukisplitting} proved that ${\s}_{\kappa}$ is small for most regular uncountable cardinals.
 He showed that for a regular $\kappa > \omega$, ${\s}_{\kappa} \geq \kappa$ iff $\kappa$ is strongly inaccessible and ${\s}_{\kappa} \geq {\kappa}^{+}$ iff $\kappa$ is weakly compact.
 The main result of Zapletal~\cite{jindrasplitting} is that if it is consistent to have a regular uncountable cardinal $\kappa$ such that ${\s}_{\kappa} \geq {\kappa}^{++}$, then it is also consistent that there is a $\kappa$ with $o(\kappa) \geq {\kappa}^{++}$.
 In particular, any positive answer to Question \ref{q:main2} would have had to start with a substantial large cardinal hypothesis.
 On the other hand, Kamo~\cite{kamo} proved that it is consistent relative to a supercompact cardinal that ${\s}_{\kappa} \geq {\kappa}^{++}$ holds at a supercompact $\kappa$.
 It was perhaps hoped that a positive answer to Question \ref{q:main2} would lead to new techniques for forcing at uncountable cardinals $\kappa$, at least when $\kappa$ is supercompact, and help generate further results like the consistency of ${\b}_{\kappa} < {\a}_{\kappa}$.
 However we will prove that Question \ref{q:main2} has a negative solution.
\begin{Theorem} \label{thm:main2}
 For any regular uncountable cardinal $\kappa$, ${\s}_{\kappa} \leq {\b}_{\kappa}$.
\end{Theorem}
It should be noted that this is not the first time that a significant difference has been observed in the behavior of cardinal characteristics between $\omega$ and bigger regular cardinals.
Blass, Hyttinen, and Zhang proved in \cite{BHZ} that ${\d}_{\kappa} = {\kappa}^{+}$ implies ${\a}_{\kappa} = {\kappa}^{+}$ for all regular uncountable cardinals $\kappa$, while the question of whether $\d = {\omega}_{1}$ implies $\a = {\omega}_{1}$ is a long-standing unresolved problem.
\section{A bound for ${\cov}^{\ast}({\ZZZ}_{0})$} \label{sec:covz0}
Theorem \ref{thm:main1} is proved in this section.
\begin{Def} \label{def:interval}
 An \emph{interval partition} is a sequence $I = \seq{i}{n}{\in}{\omega} \in \BS$ such that ${i}_{0} = 0$ and $\forall n \in \omega\[{i}_{n} < {i}_{n + 1}\]$.
 If $I$ is an interval partition and $n \in \omega$, then ${I}_{n}$ is the $n$th interval of the partition.
 In other words, ${I}_{n} = [{i}_{n}, {i}_{n + 1}) = \{k \in \omega: {i}_{n} \leq k < {i}_{n + 1}\}$.
\end{Def}
\begin{Lemma} \label{lem:regular}
 Let $I$ be an interval partition.
 Let $A \subset \omega$ be such that for each $l \geq 0$, there exists $N \in \omega$ such that for each $n \geq N$:
 \begin{enumerate}
  \item
  $\frac{\lc A \cap {I}_{n} \rc}{\lc {I}_{n} \rc} \leq {2}^{-l}$;
  \item
  $\forall i, j \in A \cap {I}_{n}\[i \neq j \implies \lc i - j \rc > {2}^{l - 1}\]$.
 \end{enumerate}
 Then $A$ has density $0$.
\end{Lemma}
\begin{proof}
 Fix $l \geq 0$.
 Using the given hypotheses, fix $N \in \omega$ such that:
 \begin{enumerate}
  \item[(3)]
  $N > 0$ and ${i}_{N} \geq {2}^{l + 1}$;
  \item[(4)]
  for each $n \geq N$:
  \begin{enumerate}
   \item[(a)]
   $\frac{\lc A \cap {I}_{n}\rc}{\lc {I}_{n} \rc} \leq {2}^{-l - 2}$;
   \item[(b)]
   $\forall i, j \in A \cap {I}_{n}\[i \neq j \implies \lc i - j \rc > {2}^{l + 1}\]$.
  \end{enumerate}
 \end{enumerate}
  Find $L \in \omega$ with $M = {i}_{L} \geq {2}^{l + 2}{i}_{N}$.
  We will show that for each $k \geq M$, $\frac{\lc A \cap k \rc}{k} \leq {2}^{-l}$.
  This will suffice to prove that $A \in {\ZZZ}_{0}$.
  To this end, we first show that for each $m \in \omega$, if ${i}_{m} \geq M$, then $\frac{\lc A \cap {i}_{m} \rc}{{i}_{m}} \leq {2}^{-l - 1}$.
  Fix any such $m$.
  Note that $m > 0$, ${i}_{N} > 0$, and that ${i}_{m} \geq {2}^{l + 2}{i}_{N} > {i}_{N}$.
  Put $Z = [{i}_{N}, {i}_{m}) = {\bigcup}_{N \leq k \leq m - 1}{I}_{k}$.
  For each $N \leq k \leq m - 1$, $\frac{\lc A \cap {I}_{k}\rc}{\lc {I}_{k} \rc} \leq {2}^{-l - 2}$, and so $\frac{\lc A \cap Z\rc}{\lc Z\rc} \leq {2}^{-l - 2}$.
  Therefore, $\frac{\lc A \cap {i}_{m} \rc}{{i}_{m}} = \frac{\lc A \cap {i}_{N}\rc}{{i}_{m}} + \frac{\lc A \cap Z \rc}{{i}_{m}} \leq \frac{{i}_{N}}{{i}_{m}} + \frac{\lc A \cap Z \rc}{\lc Z \rc}$.
  Since $\frac{{i}_{N}}{{i}_{m}} \leq {2}^{-l- 2}$, we get $\frac{\lc A \cap {i}_{m} \rc}{{i}_{m}} \leq {2}^{-l- 2} + {2}^{-l - 2} = {2}^{-l -1}$, as needed.
  
  Next, if $k \geq M$, then there exists $m \in \omega$ such that ${i}_{m} \geq M$ and $k \in {I}_{m}$.
  Thus it suffices to prove that for all $m \in \omega$, if ${i}_{m} \geq M$, then for all $k \in {I}_{m}$, $\frac{\lc A \cap k \rc}{k} \leq {2}^{-l}$.
  Fix any such $m$.
  If ${I}_{m} \cap A = 0$, then for any $k \in {I}_{m}$, $A \cap k = A \cap {i}_{m}$, and so $\frac{\lc A \cap k \rc}{k} \leq \frac{\lc A \cap k \rc}{{i}_{m}} = \frac{\lc A \cap {i}_{m}\rc}{{i}_{m}} \leq {2}^{-l - 1} \leq {2}^{-l}$.
  Thus we may assume that ${I}_{m} \cap A \neq 0$.
  Let $\{{a}_{1}, \dotsc, {a}_{p}\}$ enumerate ${I}_{m} \cap A$ in increasing order.
  Fix any $k \in {I}_{m}$.
  If $k \leq {a}_{1}$, then $A \cap {i}_{m} = A \cap k$ and so, once again $\frac{\lc A \cap k \rc}{k} \leq \frac{\lc A \cap k \rc}{{i}_{m}} = \frac{\lc A \cap {i}_{m} \rc}{{i}_{m}} \leq {2}^{-l - 1} \leq {2}^{-l}$.
  We may assume that ${a}_{1} < k$.
  Put $q = \max\{1 \leq q \leq p: {a}_{q} < k\}$ and note that $A \cap k \subset (A \cap {a}_{1}) \cup \{{a}_{1}, \dotsc, {a}_{q}\}$.
  By the remarks above, $\frac{\lc A \cap {a}_{1}\rc}{{a}_{1}} \leq {2}^{-l - 1}$, and so $\frac{\lc A \cap {a}_{1} \rc}{k} \leq \frac{\lc A \cap {a}_{1} \rc}{{a}_{1}} \leq {2}^{-l - 1}$.
  Clause (4)(b) implies that ${a}_{q} - {a}_{1} \geq (q - 1){2}^{l + 1}$ and clause (3) implies that ${2}^{l + 1} < {a}_{1}$.
  It follows from this that $\frac{q}{k} \leq {2}^{-l - 1}$.
  Therefore $\frac{\lc A \cap k \rc}{k} \leq \frac{\lc A \cap {a}_{1} \rc}{k} + \frac{q}{k} \leq {2}^{-l - 1} + {2}^{-l - 1} = {2}^{-l}$.
\end{proof}
\begin{Lemma} \label{lem:tree}
 Let $l$ be a member of $\omega$ greater than $0$ and let $X \subset \omega$ with $\lc X \rc = {2}^{l}$.
 Then there exists a sequence $\{{A}_{\sigma}: \sigma \in {2}^{\leq l}\}$ such that:
 \begin{enumerate}
  \item
  $\forall m \leq l\[{\bigcup}_{\sigma \in {2}^{m}}{A}_{\sigma} = X \wedge \forall \sigma, \tau \in {2}^{m}\[\sigma \neq \tau \implies {A}_{\sigma} \cap {A}_{\tau} = 0 \]\]$;
  \item
  $\forall \sigma \in {2}^{\leq l}\[\lc {A}_{\sigma} \rc = {2}^{l - \lc \sigma \rc}\]$ and $\forall \sigma, \tau \in {2}^{\leq l}\[\sigma \subset \tau \implies {A}_{\tau} \subset {A}_{\sigma}\]$;
  \item
  for each $\sigma \in {2}^{\leq l}$, $\forall i, j \in {A}_{\sigma}\[i \neq j \implies \lc i - j \rc > {2}^{\lc \sigma \rc - 1}\]$.
 \end{enumerate}
\end{Lemma}
\begin{proof}
 Build the ${A}_{\sigma}$ by induction on $m \leq l$.
 When $m = 0$ put ${A}_{\emptyset} = X$.
 Clause (3) is satisfied because $\forall i, j \in X\[i \neq j \implies \lc i - j \rc \geq 1 > \frac{1}{2}\]$.
 Now suppose $m < l$ and that ${A}_{\sigma}$ is given for every $\sigma \in {2}^{m}$.
 Fix any $\sigma \in {2}^{m}$.
 Then $\lc {A}_{\sigma} \rc = {2}^{l - m}$.
 Note that ${2}^{l - m} \geq 2$.
 Let $\phi: {2}^{l - m} \rightarrow {A}_{\sigma}$ be the order isomorphism.
 Put $E = \{x < {2}^{l - m}: x \ \text{is even}\}$ and $O = \{x < {2}^{l - m}: x \ \text{is odd}\}$.
 As ${2}^{l - m} \geq 2$, $\lc E \rc = \lc O \rc = {2}^{l - m - 1}$ and $E \cup O = {2}^{l - m}$.
 Define ${A}_{{\sigma}^{\frown}{\langle 0 \rangle}} = \phi'' E \subset {A}_{\sigma}$ and ${A}_{{\sigma}^{\frown}{\langle 1 \rangle}} = \phi'' O \subset {A}_{\sigma}$.
 It is easy to verify that the ${A}_{{\sigma}^{\frown}{\langle i \rangle}}$ satisfy (1)-(3), for $\sigma \in {2}^{m}$ and $i \in 2$.
\end{proof}
The next lemma is a variation on a well-known characterization of the cardinal $\d$, which may be found, for example, in \cite{blasssmall}.
We include a proof here for completeness.
\begin{Lemma} \label{lem:ipdom}
 There is a family $D$ of interval partitions such that:
 \begin{enumerate}
  \item
  $\lc D \rc \leq \d$;
  \item
  for each $I \in D$ and for each $n \in \omega$, there exists ${l}_{n} \in \omega$ such that ${l}_{n} > 0$, ${l}_{n} \geq n$, and $\lc {I}_{n} \rc = {2}^{{l}_{n}}$;
  \item
  for any interval partition $J$ there exists $I \in D$ such that $\forallbutfin n \in \omega \exists k > n \[{J}_{k} \subset {I}_{n}\]$.
 \end{enumerate}
\end{Lemma}
\begin{proof}
 Let $F \subset \BS$ be a dominating family with $\lc F \rc = \d$.
 Define an interval partition ${I}_{f}$ as follows.
 ${i}_{f, 0} = 0$.
 Given ${i}_{f, n} \in \omega$, let \begin{align*}
 M = \max\left( \left\{ f(x): x \leq {i}_{f, n} + 1 \right\} \cup \{{i}_{f, n} + 1\} \right).
 \end{align*}
 Find ${l}_{n} \in \omega$ such that ${l}_{n} > 0$, ${l}_{n} \geq n$, and ${2}^{{l}_{n}} + {i}_{f, n} > M$.
 Define ${i}_{f, n + 1} = {2}^{{l}_{n}} + {i}_{f, n}$.
 This completes the definition of the interval partition ${I}_{f}$.
 Define $D = \{{I}_{f}: f \in F\}$.
 It is clear that $\lc D \rc \leq \lc F \rc = \d$.
 And the definition of ${I}_{f}$ ensures that for each $n \in \omega$, $\lc {I}_{f, n} \rc = {2}^{{l}_{n}}$, for some ${l}_{n} \in \omega$ with ${l}_{n} > 0$ and ${l}_{n} \geq n$.
 Now suppose that $J$ is any interval partition.
 Define $g \in \BS$ as follows.
 $g(0) = 0$.
 For any $n \in \omega$, given $g(n) \in \omega$, define $g(n + 1)$ as follows.
 Put $x = g(n) + 1$ and let ${m}_{x}$ be the unique $m \in \omega$ such that $x \in {J}_{m}$.
 Let $m = \max\{{m}_{x}, n\}$ and define $g(n + 1) = {j}_{m + 2}$.
 Note that $g(n) < g(n) + 1 = x < {j}_{\left( {m}_{x} + 2 \right)} \leq {j}_{\left( m + 2 \right)} = g(n + 1)$.
 Thus $g$ is strictly increasing.
 Now using the fact that $F$ is dominating, find $f \in F$ and $N \in \omega$ such that $\forall n \geq N\[f(n) \geq g(n)\]$.
 Fix any $n \geq N$.
 Note that ${i}_{f, n} + 1 > {i}_{f, n} \geq n \geq N$.
 So $g({i}_{f, n} + 1) \leq f({i}_{f, n} + 1)$.
 Also since $g$ is strictly increasing, ${i}_{f, n} \leq g({i}_{f, n})$.
 By the definition of $g({i}_{f, n} + 1)$, there exists $m \in \omega$ such that $n \leq {i}_{f, n} \leq m$ such that ${i}_{f, n} \leq g({i}_{f, n}) < g({i}_{f, n}) + 1 < {j}_{m + 1} < {j}_{m + 2} = g({i}_{f, n} + 1) \leq f({i}_{f, n} + 1) < {i}_{f, n + 1}$.
 Thus ${J}_{m + 1} \subset {I}_{f, n}$.
 Since $m + 1 > n$, we have shown that $\forall n \geq N \exists k > n\[{J}_{k} \subset {I}_{f, n}\]$, as needed.
\end{proof}
\begin{Def} \label{def:FJ}
 Let $J$ be an interval partition such that for each $n \in \omega$ there exists ${l}_{n} \in \omega$ such that ${l}_{n} > 0$, ${l}_{n} \geq n$, and $\lc {J}_{n} \rc = {2}^{{l}_{n}}$.
 Applying Lemma \ref{lem:tree}, fix a sequence $\bar{A} = \langle {A}_{n, \sigma}: n \in \omega \wedge \sigma \in {2}^{\leq {l}_{n}}\rangle$ such that for each $n \in \omega$, the sequence $\{{A}_{n, \sigma}: \sigma \in {2}^{\leq {l}_{n}}\}$ satisfies (1)--(3) of Lemma \ref{lem:tree} with $l$ as ${l}_{n}$ and $X$ as ${J}_{n}$.
 Define ${\F}_{J, \bar{A}}$ to be the collection of all functions $f \in \BS$ such that for each $n \in \omega$ and $l < {l}_{n}$, there exists $\sigma \in {2}^{l + 1}$ such that ${f}^{-1}(\{l\}) \cap {J}_{n} = {A}_{n, \sigma}$, and there exists $\tau \in {2}^{{l}_{n}}$ such that ${f}^{-1}(\{{l}_{n}\}) \cap {J}_{n} = {A}_{n, \tau}$.
\end{Def}
Observe that if $f \in {\F}_{J, \bar{A}}$, then for each $n \in \omega$ and $k \in {J}_{n}$, $f(k) \leq {l}_{n}$.
Also for any $n, l \in \omega$,
\begin{align*}
 \displaystyle\frac{\lc \left\{ k \in {J}_{n}: f(k) \geq l \right\} \rc}{\lc {J}_{n} \rc} \leq {2}^{-l},
\end{align*}
and for any $i, j \in \left\{ k \in {J}_{n}: f(k) \geq l \right\}$, if $i \neq j$, then $\lc i - j \rc > {2}^{l - 1}$.
Moreover for any $f \in {\F}_{J, \bar{A}}$, $n \in \omega$, and $l \leq {l}_{n}$, there is ${\sigma}_{f, n, l} \in {2}^{l}$ such that ${A}_{n, {\sigma}_{f, n, l}} = \{k \in {J}_{n}: f(k) \geq l\}$.
\begin{Lemma} \label{lem:omega1rec}
Let $J$ and $\bar{A}$ be as in Definition \ref{def:FJ}.
There exists a sequence of functions $\langle {f}_{J, \bar{A}, \alpha}: \alpha < {\omega}_{1} \rangle$ such that:
\begin{enumerate}
 \item
 for all $\alpha < {\omega}_{1}$, ${f}_{J, \bar{A}, \alpha} \in {\F}_{J, \bar{A}}$;
 \item
 for each triple $\langle i, m, F \rangle$ such that $i, m \in \omega$, $m \leq {2}^{i}$, and $F \in {\[{\omega}_{1}\]}^{m}$, there exists ${B}_{i, m, F} \in {\[\left( \omega \setminus i \right)\]}^{\omega}$ such that 
 \begin{align*}
 \forall \alpha, \beta \in F \forall n \in {B}_{i, m, F}\[\alpha \neq \beta \implies {\sigma}_{{f}_{J, \bar{A}, \alpha}, n, i} \neq {\sigma}_{{f}_{J, \bar{A}, \beta}, n, i}\].
 \end{align*}
\end{enumerate}
\end{Lemma}
\begin{proof}
 We write ${\sigma}_{\alpha, n, i}$ instead of ${\sigma}_{{f}_{J, \bar{A}, \alpha}, n, i}$ and ${f}_{\alpha}$ instead of ${f}_{J, \bar{A}, \alpha}$ to simplify the notation.
 We first define ${B}_{i, 0, \emptyset} = \omega \setminus i$, for all $i \in \omega$.
 Fix $\alpha < {\omega}_{1}$.
 Suppose that ${f}_{\xi}$ has been defined for all $\xi < \alpha$.
 Suppose also that ${B}_{i, m, F}$ has been defined for all triples $\langle i, m, F \rangle$ such that $i, m \in \omega$, $m \leq {2}^{i}$, and $F \in {\[\alpha\]}^{m}$.
 Let $\{\langle {i}_{n}, {m}_{n}, {F}_{n} \rangle: n \in \omega\}$ be a 1-1 enumeration of all triples $\langle i, m, F \rangle$ such that $i, m \in \omega$, $m < {2}^{i}$, and $F \in {\[\alpha\]}^{m}$.
 Let ${B}_{n}$ denote ${B}_{{i}_{n}, {m}_{n}, {F}_{n}}$.
 Find a sequence $\seq{C}{n}{\in}{\omega}$ such that ${C}_{n} \in {\[{B}_{n}\]}^{\omega}$ and $\forall n < m < \omega\[{C}_{n} \cap {C}_{m} = 0\]$.
 For each $\langle i, m, F \rangle$ with $i, m \in \omega$, $m < {2}^{i}$, and $F \in {\[\alpha\]}^{m}$, define ${B}_{i, m + 1, F \cup \{\alpha\}} = {C}_{n}$, where $n$ is the unique $n \in \omega$ such that $\langle {i}_{n}, {m}_{n}, {F}_{n} \rangle = \langle i, m, F \rangle$.
 Note that ${B}_{i, m + 1, F \cup \{\alpha\}} \in {\[\left( \omega \setminus i \right)\]}^{\omega}$.
 We now define ${f}_{\alpha} \in {\F}_{J, \bar{A}}$.
 Fix $k \in \omega$.
 Suppose first that $k \in {\bigcup}_{n \in \omega}{{C}_{n}}$.
 Let $n$ be the unique $n \in \omega$ such that $k \in {C}_{n}$.
 Since ${i}_{n} \leq k \leq {l}_{k}$, ${\sigma}_{\xi, k, {i}_{n}}$ is defined and belongs to ${2}^{{i}_{n}}$, for each $\xi \in {F}_{n}$.
 So $\{{\sigma}_{\xi, k, {i}_{n}}: \xi \in {F}_{n}\}$ is a subset of ${2}^{{i}_{n}}$ with cardinality less than ${2}^{{i}_{n}}$.
 So choose $\eta \in {2}^{{l}_{k}}$ such that $\eta \restrict {i}_{n} \notin \{{\sigma}_{\xi, k, {i}_{n}}: \xi \in {F}_{n}\}$.
 If $k \notin {\bigcup}_{n \in \omega}{{C}_{n}}$, then choose $\eta \in {2}^{{l}_{k}}$ to be arbitrary.
 In either case, for each $l < {l}_{k}$, define ${\tau}_{l} = {\left( \eta \restrict l \right)}^{\frown}{\langle 1 - \eta(l) \rangle}\in {2}^{l + 1}$ and define ${\tau}_{{l}_{k}} = \eta$.
 It is clear that for all $l < l' \leq {l}_{k}$, ${A}_{k, {\tau}_{l}} \cap {A}_{k, {\tau}_{l'}} = 0$ and that ${\bigcup}_{l\leq{l}_{k}}{{A}_{k, {\tau}_{l}}} = {J}_{k}$.
 Define ${f}_{\alpha}''{A}_{k, {\tau}_{l}} = \{l\}$, for each $l \leq {l}_{k}$.
 This completes the definition of ${f}_{\alpha}$.
 In the case when $k \in {C}_{n}$ for some $n \in \omega$, ${\sigma}_{\alpha, k, {i}_{n}} = \eta \restrict {i}_{n}$.
 One sees that ${f}_{\alpha} \in {\F}_{J, \bar{A}}$ and that for each $n \in \omega$, $\xi \in {F}_{n}$, and $k \in {C}_{n}$, ${\sigma}_{\alpha, k, {i}_{n}} \neq {\sigma}_{\xi, k, {i}_{n}}$.
 Now to verify clause (2) after stage $\alpha$ of the construction, suppose that $\langle i, m, F \rangle$ is a triple such that $i, m \in \omega$, $m \leq {2}^{i}$, and $F \in {\[\alpha + 1\]}^{m}$.
 If $F \in {\[\alpha\]}^{m}$, then clause (2) holds by the induction hypothesis.
 So assume that $\alpha \in F$ and let $G = F \setminus \{\alpha \}$.
 Let $n$ be the unique element of $\omega$ such that $\langle {i}_{n}, {m}_{n}, {F}_{n} \rangle = \langle i, m - 1, G \rangle$.
 Then ${B}_{i, m, F} = {C}_{n} \subset {B}_{n} = {B}_{i, m - 1, G}$.
 Take $\xi, \zeta \in F$ and $k \in {B}_{i, m, F}$.
 Suppose $\xi < \zeta$.
 If $\xi, \zeta \in G$, then since $k \in {B}_{i, m - 1, G}$, ${\sigma}_{\xi, k, i} \neq {\sigma}_{\zeta, k, i}$ holds by the induction hypothesis.
 If $\zeta = \alpha$, then ${\sigma}_{\zeta, k, i} = {\sigma}_{\alpha, k, i} = {\sigma}_{\alpha, k, {i}_{n}} \neq {\sigma}_{\xi, k, {i}_{n}} = {\sigma}_{\xi, k, i}$.
 This finishes the construction.
\end{proof}
Note that if $m = {2}^{i}$ and $F \in {\[{\omega}_{1}\]}^{m}$, then for any $k \in {B}_{i, m, F}$, $\{{\sigma}_{\alpha, k, i}: \alpha \in F\}$ is a subset of ${2}^{i}$ of size ${2}^{i}$.
Therefore ${\bigcup}_{\alpha \in F}{{A}_{k, {\sigma}_{\alpha, k, i}}} = {\bigcup}_{\sigma \in {2}^{i}}{A}_{k, \sigma} = {J}_{k}$.
This consequence of Lemma \ref{lem:omega1rec} will be used below.
\begin{Def} \label{def:thefamily}
 Let $D$ be a family of interval partitions as in Lemma \ref{lem:ipdom}. For each $J \in D$ fix a sequence ${\bar{A}}_{J} = \langle {A}_{J, k, \sigma}: k \in \omega \wedge \sigma \in {2}^{\leq{l}_{J, k}}\rangle$ as in Definition \ref{def:FJ}.
 Use Lemma \ref{lem:omega1rec} to fix a sequence of functions $\langle {f}_{J, {\bar{A}}_{J}, \alpha}: \alpha < {\omega}_{1}\rangle$ satisfying (1) and (2) of Lemma \ref{lem:omega1rec}.
 We will write ${f}_{J, \alpha}$ instead of ${f}_{J, {\bar{A}}_{J}, \alpha}$.
 For any tuple $\langle I, J, \alpha, l \rangle \in D \times D \times {\omega}_{1} \times \omega$, let ${Z}_{I, J, \alpha, l} = {\bigcup}_{k \in {I}_{l}}{\{x \in {J}_{k}: {f}_{J, \alpha}(x) \geq l \}}$.
 For each triple $\langle I, J, \alpha \rangle \in D \times D \times {\omega}_{1}$ define ${Z}_{I, J, \alpha} = {\bigcup}_{l \in \omega}{{Z}_{I, J, \alpha, l}} \subset \omega$.
 Note that the family $\{{Z}_{I, J, \alpha}: \langle I, J, \alpha \rangle \in D \times D \times {\omega}_{1}\}$ has size at most $\d$.
\end{Def}
\begin{Lemma} \label{label:family0}
 ${Z}_{I, J, \alpha} \in {\ZZZ}_{0}$ for every triple $\langle I, J, \alpha \rangle \in D \times D \times {\omega}_{1}$.
\end{Lemma}
\begin{proof}
 We apply Lemma \ref{lem:regular} with $J$ as $I$ and ${Z}_{I, J, \alpha}$ as $A$.
 Fix $l \geq 0$.
 Let $N = {i}_{l}$ and let $k \geq N$ be given.
 Let ${l}^{\ast}$ be the unique member of $\omega$ such that $k \in {I}_{{l}^{\ast}}$.
 Note that $l \leq {l}^{\ast}$.
 Now ${Z}_{I, J, \alpha} \cap {J}_{k} = \{x \in {J}_{k}: {f}_{J, \alpha}(x)\geq{l}^{\ast}\}$.
 Since ${f}_{J, \alpha} \in {\F}_{J, {\bar{A}}_{J}}$, $\frac{\lc {Z}_{I, J, \alpha} \cap {J}_{k} \rc}{\lc {J}_{k} \rc} \leq {2}^{-{l}^{\ast}} \leq {2}^{-l}$ and for any $i, j \in {Z}_{I, J, \alpha} \cap {J}_{k}$, if $i \neq j$, then $\lc i - j \rc > {2}^{{l}^{\ast} - 1} \geq {2}^{l - 1}$, exactly as needed. 
\end{proof}
\begin{Lemma} \label{lem:bound}
 Fix $J \in D$ and $\alpha \in {\omega}_{1}$.
 Suppose $A \subset \omega$ and suppose that for every $I \in D$, $A \cap {Z}_{I, J, \alpha}$ is finite.
 Then there exists $i \in \omega$ such that ${f}_{J, \alpha}''A \subset i$.
\end{Lemma}
\begin{proof}
 Suppose not.
 Then for every $i \in \omega$, there exists $x \in A$ such that ${f}_{J, \alpha}(x) \geq i$.
 Define an interval partition $K$ as follows.
 ${k}_{0} = 0$.
 Given ${k}_{n}$, let $U = {f}_{J, \alpha}''\left( {\bigcup}_{k < {k}_{n}}{{J}_{k}}\right)$.
 Let $i = \max(U \cup \{n\}) + 1$.
 Find $x \in A$ such that ${f}_{J, \alpha}(x) \geq i$.
 Let ${k}^{\ast} \in \omega$ be such that $x \in {J}_{{k}^{\ast}}$.
 Note that ${k}_{n} \leq {k}^{\ast}$.
 Set ${k}_{n + 1} = {k}^{\ast} + 1$.
 This finishes the construction of $K$.
 By construction we have that $\forall n \in \omega \exists k \in {K}_{n} \exists x \in {J}_{k} \cap A \[{f}_{J, \alpha}(x) > n\]$.
 Find $I \in D$ and $N \in \omega$ such that $\forall l \geq N\exists n > l\[{K}_{n} \subset {I}_{l}\]$.
 It follows that $\forall l \geq N\[{Z}_{I, J, \alpha, l} \cap A \neq 0\]$.
 This contradicts the hypothesis that ${Z}_{I, J, \alpha} \cap A$ is finite.
\end{proof}
\begin{proof}[Proof of Theorem \ref{thm:main1}]
 The family $\{{Z}_{I, J, \alpha}: \langle I, J, \alpha \rangle \in D \times D \times {\omega}_{1}\}$ is a subset of ${\ZZZ}_{0}$ of cardinality at most $\d$.
 Suppose for a contradiction that $A \in \cube$ and that $A \cap {Z}_{I, J, \alpha}$ is finite for every $\langle I, J, \alpha \rangle \in D \times D \times {\omega}_{1}$.
 Fix $J \in D$ such that $\forallbutfin k \in \omega \[A \cap {J}_{k} \neq 0\]$.
 By Lemma \ref{lem:bound}, $\forall \alpha < {\omega}_{1} \exists {i}_{\alpha} \in \omega\[{f}_{J, \alpha}''A \subset {i}_{\alpha}\]$.
 There exist $i \in \omega$ and $S \in {\[{\omega}_{1}\]}^{{\omega}_{1}}$ such that $\forall \alpha \in S\[i = {i}_{\alpha}\]$.
 Let $m = {2}^{i}$ and $F \in {\[S\]}^{m}$.
 By the remark following Lemma \ref{lem:omega1rec}, there exists ${B}_{J, {\bar{A}}_{J}, i, m, F} \in \cube$ such that for every $k \in {B}_{J, {\bar{A}}_{J}, i, m, F}$, ${\bigcup}_{\alpha \in F}{ {A}_{J, k, {\sigma}_{J, {\bar{A}}_{J}, \alpha, k, i} } } = {J}_{k}$, where $\{x \in {J}_{k}: {f}_{J, \alpha}(x) \geq i\} = {A}_{J, k, {\sigma}_{J, {\bar{A}}_{J}, \alpha, k, i} }$.
 It follows that for each $k \in {B}_{J, {\bar{A}}_{J}, i, m, F}$, $A \cap {J}_{k} = 0$.
 However this contradicts the choice of $J$.
\end{proof}
\begin{Cor} \label{cor:main1}
Let $\V$ be any ground model and let $E \in \V$ be a dominating family of minimal size.
If $\P \in \V$ diagonalizes ${\ZZZ}_{0} \cap \V$, then $E$ is no longer a dominating family in ${\V}^{\P}$.
\end{Cor}
The proof of Theorem \ref{thm:main1} can be adapted to cover several density ideals (see Definition 1.6 of \cite{michaelstarinv} for the exact definition of a density ideal).
By the results in \cite{michaelstarinv}, ${\cov}^{\ast}({\ZZZ}_{0}) \leq {\cov}^{\ast}(\I)$ for every density ideal $\I$.
Solecki and Todorcevic~\cite{soleckitodorcevic2} have introduced the more general notion of a density-like ideal.
\begin{Def} \label{def:densitylike}
 Let $\varphi: {2}^{\omega} \rightarrow \[0, \infty\]$ be a lower semi-continuous sub-measure.
 Let $\I = \mathrm{Exh}(\varphi) = \{X \subset \omega: {\lim}_{n \rightarrow \infty}\varphi(X \setminus n) = 0\}$.
 $\I$ is said to be \emph{density-like} if for every $\epsilon > 0$ there exists $\delta > 0$ such that for every sequence $\seq{F}{n}{\in}{\omega}$ of finite sets, if $\forall n \in \omega \[\varphi({F}_{n}) < \delta\]$, then there exists $a \in \cube$ such that $\varphi({\bigcup}_{n \in a}{{F}_{n}}) < \epsilon$.
\end{Def}
It would be of interest to see whether a similar bound on ${\cov}^{\ast}(\I)$ can be proved for all density-like ideals $\I$.
\section{The bounding and splitting numbers at uncountable cardinals} \label{sec:skbk}
 Theorem \ref{thm:main2} is proved in this section.
 We begin with an observation due to Suzuki~\cite{suzukisplitting}, which shows that $\kappa$ must be a large cardinal if ${\s}_{\kappa}$ is to be big and $\kappa$ is to be regular and uncountable.
 Though we will not need this observation, we include a proof below for completeness.
 Its converse is also true and was noted by Suzuki.
 \begin{Lemma}[Suzuki]\label{lem:wc}
  Let $\kappa > \omega$ be a regular cardinal.
  If ${\s}_{\kappa} > \kappa$, then $\kappa$ is weakly compact.
 \end{Lemma}
 \begin{proof}
  We show that $\kappa \rightarrow {(\kappa)}^{2}_{2}$.
  Let $c: {\[\kappa\]}^{2} \rightarrow 2$ be a coloring.
  For each $\alpha < \kappa$ and $i \in 2$, let ${K}_{\alpha, i} = \{\beta > \alpha: c(\{\alpha, \beta\}) = i\}$.
  Since $\{{K}_{\alpha, 0}: \alpha < \kappa\}$ is not a splitting family, there exists $x \in {\[\kappa\]}^{\kappa}$ such that $\forall \alpha < \kappa\exists {i}_{\alpha} \in 2\[x \; {\subset}^{\ast} \; {K}_{\alpha, {i}_{\alpha}}\]$.
  Find $y \in {\[x\]}^{\kappa}$ and $i \in 2$ such that $\forall \gamma \in y\[{i}_{\gamma} = i\]$.
  Define a sequence $\langle {\gamma}_{\alpha}: \alpha < \kappa \rangle \subset y$ by induction on $\alpha < \kappa$ as follows.
  Fix $\alpha < \kappa$ and suppose that $\seq{\gamma}{\xi}{<}{\alpha} \subset y$ is given.
  For each $\xi < \alpha$, there exists ${\delta}_{\xi} < \kappa$ such that $x \setminus {\delta}_{\xi} \subset {K}_{{\gamma}_{\xi}, i}$.
  Put $\delta = \sup\{{\delta}_{\xi}: \xi < \alpha\} < \kappa$ and ${\gamma}_{\alpha} = \min(y \setminus \delta)$.
  It is clear that for each $\xi < \alpha$, ${\gamma}_{\xi} < {\gamma}_{\alpha}$ and $c\left( \left\{ {\gamma}_{\xi}, {\gamma}_{\alpha} \right\} \right) = i$.
  Therefore $\{{\gamma}_{\alpha}: \alpha < \kappa\}$ is a homogeneous set of cardinality $\kappa$ for the color $i$.
 \end{proof}
 \begin{Def} \label{def:modD}
  Let $\kappa > \omega$ be regular and suppose that there exists a cardinal $\lambda$ such that $\kappa < \lambda < {\s}_{\kappa}$.
  Fix a sufficiently large regular cardinal $\theta$ ($\theta = {\left( {2}^{{2}^{{\s}_{\kappa}}} \right)}^{+}$ will suffice).
  Let $M \prec H(\theta)$ be such that $\lambda \subset M$ and $\lc M \rc = \lambda$.
  $M \cap \Pset(\kappa)$ is not a splitting family.
  So there exists ${A}_{\ast} \in {\[\kappa\]}^{\kappa}$ such that for all $x \in M \cap \Pset(\kappa)$ either ${A}_{\ast} \; {\subset}^{\ast} \; \left( \kappa \setminus x \right)$ or ${A}_{\ast} \; {\subset}^{\ast} \; x$.
  Define $D$ to be $\{x \in \Pset(\kappa): {A}_{\ast} \; {\subset}^{\ast} \; x \}$.
  For any $f, g \in M \cap {\kappa}^{\kappa}$, define $f \; {\sim}_{D} \; g$ iff $\{\alpha < \kappa: f(\alpha) = g(\alpha)\} \in D$.
  This is an equivalence relation on $M \cap {\kappa}^{\kappa}$.
  For $f \in M \cap {\kappa}^{\kappa}$, let ${\[f\]}_{D} = \{g \in M \cap {\kappa}^{\kappa}: f \; {\sim}_{D} \; g\}$.
  For $f, g \in M \cap {\kappa}^{\kappa}$, define ${\[f\]}_{D}\; {<}_{D} \; {\[g\]}_{D}$ iff $\{\alpha < \kappa: f(\alpha) < g(\alpha)\} \in D$.
  Let $L = \{{\[f\]}_{D}: f \in M \cap {\kappa}^{\kappa}\}$.
  
  Let $i: \kappa \rightarrow \kappa$ be the identity function on $\kappa$.
  For each $\alpha \in \kappa$, let ${c}_{\alpha}: \kappa \rightarrow \kappa$ be the function such that ${c}_{\alpha}(\beta) = \alpha$, for all $\beta \in \kappa$.
  Note that $i, {c}_{\alpha} \in M \cap {\kappa}^{\kappa}$, for all $\alpha \in \kappa$.
 \end{Def}
 We observe that $D$ is a $\kappa$-complete filter on $\kappa$.
 Also if $X \in D$, then $\lc X \rc = \kappa$.
 Next we note that for any $\delta < \kappa$, $\left( \kappa \setminus \delta \right) \in D$.
 Finally if $0 < \delta < \kappa$ and $\seq{X}{\alpha}{<}{\delta} \in M$ is a partition of $\kappa$, then there is a unique $\alpha < \delta$ such that ${X}_{\alpha} \in D$.
 To see this note that $\delta \subset \kappa \subset \lambda \subset M$, and so $\{{X}_{\alpha}: \alpha < \delta\} \subset M$.
 For each $\alpha < \delta$ there exists ${i}_{\alpha} \in 2$ such that ${A}_{\ast} \; {\subset}^{\ast} \; {X}^{{i}_{\alpha}}_{\alpha}$, where ${X}^{0}_{\alpha} = {X}_{\alpha}$ and ${X}^{1}_{\alpha} = \kappa \setminus {X}_{\alpha}$.
 Thus $\{{X}^{{i}_{\alpha}}_{\alpha}: \alpha < \delta\} \subset D$, and by the $\kappa$-completeness of $D$, $X = \bigcap \{{X}^{{i}_{\alpha}}_{\alpha}: \alpha < \delta\} \in D$.
 So $\lc X \rc = \kappa$.
 By hypothesis, if $\alpha < \beta < \delta$, then ${X}_{\alpha} \cap {X}_{\beta} = 0$, and also $\bigcap \{{X}^{1}_{\alpha}: \alpha < \delta\} = 0$.
 It follows that ${i}_{\alpha} = 0$ for some unique $\alpha < \delta$.
 \begin{Lemma} \label{lem:sup}
  The structure $\pr{L}{{<}_{D}}$ is a linear order.
  Moreover $\{{\[{c}_{\alpha}\]}_{D}: \alpha < \kappa \}$ has a least upper bound in $L$.
 \end{Lemma}
 \begin{proof}
  The relation ${<}_{D}$ is transitive because $D$ is a filter.
  Given $f, g \in M \cap {\kappa}^{\kappa}$, the sets $\{\alpha < \kappa: f(\alpha) = g(\alpha)\}$, $\{\alpha < \kappa: f(\alpha) < g(\alpha)\}$, and $\{\alpha < \kappa: g(\alpha) < f(\alpha)\}$ all belong to $M$ and they partition $\kappa$.
  So by the remarks above, exactly one of them belongs to $D$, whence exactly one of ${\[f\]}_{D} = {\[g\]}_{D}$, ${\[f\]}_{D} \; {<}_{D} \; {\[g\]}_{D}$, or ${\[g\]}_{D} \; {<}_{D} \; {\[f\]}_{D}$ holds.
  For the second statement note that ${\[i\]}_{D} \in L$ and is an upper bound of $\{{\[{c}_{\alpha}\]}_{D}: \alpha < \kappa\}$.
  If there is no least upper bound, then we can get a sequence $\seq{f}{n}{\in}{\omega} \subset M \cap {\kappa}^{\kappa}$ such that for each $n \in \omega$, ${\[{f}_{n}\]}_{D} \in L$ is an upper bound of $\{{\[{c}_{\alpha}\]}_{D}: \alpha < \kappa \}$ and ${\[{f}_{n + 1}\]}_{D} \; {<}_{D} \; {\[{f}_{n}\]}_{D}$.
  Thus for each $n \in \omega$, ${X}_{n} = \{\beta < \kappa: {f}_{n + 1}(\beta) < {f}_{n}(\beta)\} \in D$.
  By $\kappa$-completeness of $D$, $X = {\bigcap}_{n \in \omega}{{X}_{n}} \in D$, and so $X \neq 0$.
  Choosing $\beta \in X$, we get an infinite descending sequence of ordinals ${f}_{0}(\beta) > {f}_{1}(\beta) > \dotsb$, which is a contradiction.
 \end{proof}
 Of course the argument of Lemma \ref{lem:sup} shows that $\pr{L}{{<}_{D}}$ is a well-order.
 But we will not need this in what follows.
 One can also take the reduced power of the structure $\pr{M}{\in}$ with respect to the filter $D$.
 The above argument shows that this structure will be well-founded.
 It also possible to argue for Theorem \ref{thm:main2} in terms of the resulting embedding, which may not be elementary.
 Similar ideas were used by Zapletal in \cite{jindrasplitting} to prove his result there that the statement ${\s}_{\kappa} > {\kappa}^{+}$ has large consistency strength.
 The proof we give below avoids dealing with the reduced power of $\pr{M}{\in}$.
 \begin{Def} \label{def:fast}
  Fix a function ${f}_{\ast} \in M \cap {\kappa}^{\kappa}$ such that ${\[{f}_{\ast}\]}_{D} \in L$ is a least upper bound of $\{{\[{c}_{\alpha}\]}_{D}: \alpha < \kappa\}$.  
 \end{Def}
 \begin{Lemma} \label{lem:normal}
  If $C \in M$ is a club in $\kappa$, then ${f}^{-1}_{\ast}(C) \in D$.
 \end{Lemma}
 \begin{proof}
  Suppose for a contradiction that ${f}^{-1}_{\ast}(C) \notin D$.
  Since ${f}_{\ast}, C, \kappa \in M$, both ${f}^{-1}_{\ast}(C)$ and $\kappa \setminus {f}^{-1}_{\ast}(C) = {f}^{-1}_{\ast}(\kappa \setminus C)$ belong to $M$ and partition $\kappa$.
  Therefore, ${X}_{0} = {f}^{-1}_{\ast}(\kappa \setminus C) \in D$.
  Next since ${\[{c}_{0}\]}_{D} \; {<}_{D} \; {\[{f}_{\ast}\]}_{D}$, ${X}_{1} = \{\beta < \kappa: 0 < {f}_{\ast}(\beta)\} \in D$.
  Thus $X = {X}_{0} \cap {X}_{1} \in M \cap D$.
  Define a function $f: \kappa \rightarrow \kappa$ as follows.
  For any $\alpha \in \kappa$, if $\alpha \in X$, then let $f(\alpha) = \sup(C \cap {f}_{\ast}(\alpha))$, and let $f(\alpha) = 0$ otherwise.
  If $\alpha \in X$, then since ${f}_{\ast}(\alpha) \notin C$, $f(\alpha)=\sup(C \cap {f}_{\ast}(\alpha)) < {f}_{\ast}(\alpha)$.
  It is clear that $f \in M$.
  Thus ${\[f\]}_{D} \in L$ and ${\[f\]}_{D} \; {<}_{D} \; {\[{f}_{\ast}\]}_{D}$.
  On the other hand consider any $\alpha < \kappa$.
  Fix $\delta \in C$ with $\delta > \alpha$.
  Since ${\[{c}_{\delta}\]}_{D}\; {<}_{D} \; {\[{f}_{\ast}\]}_{D}$, $Y = \{\beta < \kappa: \delta < {f}_{\ast}(\beta)\} \in D$.
  So $Z = X \cap Y \in D$ and $\forall \beta \in Z\[{c}_{\alpha}(\beta) = \alpha < f(\beta)\]$, whence ${\[{c}_{\alpha}\]}_{D} \; {<}_{D} \; {\[f\]}_{D}$.
  However this contradicts the choice of ${f}_{\ast}$.
 \end{proof}
 \begin{Lemma} \label{lem:bounded}
 $M \cap {\kappa}^{\kappa}$ is bounded.
 \end{Lemma}
 \begin{proof}
 First note that ${f}_{\ast}''{A}_{\ast}$ is an unbounded subset of $\kappa$.
 This is because for any $\alpha < \kappa$, $\{\beta < \kappa: \alpha < {f}_{\ast}(\beta)\} \in D$, whence ${A}_{\ast} \; {\subset}^{\ast} \; \{\beta < \kappa: \alpha < {f}_{\ast}(\beta)\}$.
 So we can find $\beta \in {A}_{\ast}$ with $\alpha < {f}_{\ast}(\beta)$ because $\lc {A}_{\ast} \rc = \kappa$.
 Next if $C \in M$ is any club in $\kappa$, then by Lemma \ref{lem:normal}, ${A}_{\ast} \; {\subset}^{\ast} \; {f}^{-1}_{\ast}(C)$.
 It follows that ${f}_{\ast}''{A}_{\ast} \; {\subset}^{\ast} \; C$.
 Now since $\kappa$ is regular, $\otp({f}_{\ast}''{A}_{\ast}) = \kappa$.
 Let $g: \kappa \rightarrow {f}_{\ast}''{A}_{\ast}$ be the unique order isomorphism.
 Define $h: \kappa \rightarrow \kappa$ by $h(\alpha) = g(\alpha + 1)$, for each $\alpha \in \kappa$.
 We claim that $h$ bounds $M \cap {\kappa}^{\kappa}$.
 Indeed, let $f \in M \cap {\kappa}^{\kappa}$.
 Then ${C}_{f} = \{\alpha < \kappa: \alpha \ \text{is closed under} \ f\} \in M$ is a club in $\kappa$.
 Therefore there exists $\delta < \kappa$ such that $\left( {f}_{\ast}''{A}_{\ast} \right) \setminus \delta \subset {C}_{f}$.
 Consider any $\alpha < \kappa$ with $\alpha \geq \delta$.
 Then $h(\alpha) = g(\alpha + 1) \in {f}_{\ast}''{A}_{\ast}$, and since $g$ is order preserving, $\delta \leq \alpha < \alpha + 1 \leq g(\alpha + 1)$.
 Therefore $g(\alpha + 1) \in {C}_{f}$, and since $\alpha < g(\alpha + 1)$, $f(\alpha) < g(\alpha + 1) = h(\alpha)$.
 Thus we have shown that $\forall \alpha < \kappa\[\alpha \geq \delta \implies f(\alpha) < h(\alpha)\]$, as required.
 \end{proof}
\begin{proof}[Proof of Theorem \ref{thm:main2}]
Suppose for a contradiction that ${\b}_{\kappa} < {\s}_{\kappa}$.
Put $\lambda = {\b}_{\kappa}$.
By the results in \cite{sh541}, $\kappa < \lambda < {\s}_{\kappa}$.
Let $\{{f}_{\xi}: \xi < \lambda\} \subset {\kappa}^{\kappa}$ be an unbounded family of size $\lambda$.
Let $M \prec H(\theta)$ be such that $\lc M \rc = \lambda$ and $\lambda \cup \{{f}_{\xi}: \xi < \lambda\} \subset M$.
Applying Lemma \ref{lem:bounded}, we get that $\{{f}_{\xi}: \xi < \lambda\} \subset M \cap {\kappa}^{\kappa}$ is bounded, which is a contradiction.
\end{proof}
We remark that it is not hard to force the consistency of ${\s}_{\kappa} < {\b}_{\kappa}$ for most regular uncountable cardinals $\kappa$.
Indeed if $\omega < \kappa = {\kappa}^{< \kappa}$ and if $\lambda > {2}^{\kappa}$ is a regular cardinal, then we can do a $\left(\mathord{< \kappa}\right)$--support iteration $\langle {\P}_{\alpha}, {\mathring{\Q}}_{\alpha}: \alpha < {\lambda}^{+}\rangle$ such that if $\alpha < {\kappa}^{+}$, then ${\mathring{\Q}}_{\alpha}$ names the poset for adding a Cohen subset of $\kappa$, while if ${\kappa}^{+} \leq \alpha < {\lambda}^{+}$, then ${\mathring{\Q}}_{\alpha}$ names the poset for adding a dominating function from $\kappa$ to $\kappa$.
It is straightforward to check that in the resulting model, ${\s}_{\kappa} \leq {\kappa}^{+} < \lambda < {\b}_{\kappa}$.
\section{Some questions}\label{sec:q}
It is unknown to what extent the bound on ${\cov}^{\ast}({\ZZZ}_{0})$ given by Theorem \ref{thm:main1} can be improved.
\begin{Question} \label{q:1}
 Is ${\cov}^{\ast}({\ZZZ}_{0}) \leq \b$?
\end{Question}
This is essentially equivalent to asking whether there is a proper forcing that diagonalizes $\V \cap {\ZZZ}_{0}$ while preserving all unbounded families in $\V$.

The following is an outstanding open problem about cardinal invariants above the continuum.
\begin{Question} \label{q:2}
 Is it consistent to have a regular uncountable cardinal $\kappa$ for which ${\b}_{\kappa} < {\a}_{\kappa}$?
\end{Question}
\def\polhk#1{\setbox0=\hbox{#1}{\ooalign{\hidewidth
  \lower1.5ex\hbox{`}\hidewidth\crcr\unhbox0}}}
\providecommand{\bysame}{\leavevmode\hbox to3em{\hrulefill}\thinspace}
\providecommand{\MR}{\relax\ifhmode\unskip\space\fi MR }
\providecommand{\MRhref}[2]{%
  \href{http://www.ams.org/mathscinet-getitem?mr=#1}{#2}
}
\providecommand{\href}[2]{#2}


\begin{thebibliography}{10}

\bibitem{barsmall}
T.~Bartoszy{\'n}ski, \emph{Invariants of measure and category}, Handbook of set
  theory. {V}ols. 1, 2, 3, Springer, Dordrecht, 2010, pp.~491--555.

\bibitem{blasssmall}
A.~Blass, \emph{Combinatorial cardinal characteristics of the continuum},
  Handbook of set theory. {V}ols. 1, 2, 3, Springer, Dordrecht, 2010,
  pp.~395--489.

\bibitem{BHZ}
A.~Blass, T.~Hyttinen, and Y.~Zhang, \emph{Mad families and their neighbours},
  (Preprint).

\bibitem{bs1}
J.~Brendle and S.~Shelah, \emph{Ultrafilters on {$\omega$}---their ideals and
  their cardinal characteristics}, Trans. Amer. Math. Soc. \textbf{351} (1999),
  no.~7, 2643--2674.

\bibitem{sh541}
J.~Cummings and S.~Shelah, \emph{Cardinal invariants above the continuum}, Ann.
  Pure Appl. Logic \textbf{75} (1995), no.~3, 251--268.

\bibitem{uk2k1}
M.~D{\v{z}}amonja and S.~Shelah, \emph{Universal graphs at the successor of a
  singular cardinal}, J. Symbolic Logic \textbf{68} (2003), no.~2, 366--388.

\bibitem{frem3}
D.~H. Fremlin, \emph{Measure theory. {V}ol. 5}, Torres Fremlin, Colchester,
  2015, Set-theoretic measure theory. {P}art {I}, {I}{I}.

\bibitem{uk2k2}
S.~Garti and S.~Shelah, \emph{Partition calculus and cardinal invariants}, J.
  Math. Soc. Japan \textbf{66} (2014), no.~2, 425--434.

\bibitem{michaelstarinv}
F.~Hern{\'a}ndez-Hern{\'a}ndez and M.~Hru{\v{s}}{\'a}k, \emph{Cardinal
  invariants of analytic {$P$}-ideals}, Canad. J. Math. \textbf{59} (2007),
  no.~3, 575--595.

\bibitem{kamo}
S.~Kamo, \emph{Splitting numbers on uncountable regular cardinals},
  (Preprint).

\bibitem{zapping}
C.~Laflamme, \emph{Zapping small filters}, Proc. Amer. Math. Soc. \textbf{114}
  (1992), no.~2, 535--544.

\bibitem{b<s}
S.~Shelah, \emph{On cardinal invariants of the continuum}, Axiomatic set theory
  ({B}oulder, {C}olo., 1983), Contemp. Math., vol.~31, Amer. Math. Soc.,
  Providence, RI, 1984, pp.~183--207.

\bibitem{soleckitodorcevic2}
S.~Solecki and S.~Todorcevic, \emph{Avoiding families and {T}ukey functions on
  the nowhere-dense ideal}, J. Inst. Math. Jussieu \textbf{10} (2011), no.~2,
  405--435.

\bibitem{suzukisplitting}
T.~Suzuki, \emph{About splitting numbers}, Proc. Japan Acad. Ser. A Math. Sci.
  \textbf{74} (1998), no.~2, 33--35.

\bibitem{analyticgaps}
S.~Todor{\v{c}}evi{\'c}, \emph{Analytic gaps}, Fund. Math. \textbf{150} (1996),
  no.~1, 55--66.

\bibitem{jindrasplitting}
J.~Zapletal, \emph{Splitting number at uncountable cardinals}, J. Symbolic
  Logic \textbf{62} (1997), no.~1, 35--42.

\end{thebibliography}
\end{document}